\documentclass{article}%
\usepackage{amsmath}%
\usepackage{amsfonts}%
\usepackage{amssymb}%
\usepackage{graphicx}
\providecommand{\U}[1]{\protect\rule{.1in}{.1in}}
\newtheorem{theorem}{Theorem}
\newtheorem{acknowledgement}[theorem]{Acknowledgement}

\newtheorem{lemma}[theorem]{Lemma}

\newenvironment{proof}[1][Proof]{\noindent\textbf{#1.} }{\ \rule{0.5em}{0.5em}}

\begin{document}

\title{Holomorphic Continuation via Laplace-Fourier series}
\author{O. Kounchev and H. Render}
\maketitle
\begin{abstract}
Let $B_{R}$ be the ball in the euclidean space $\mathbb{R}^{n}$ with center
$0$ and radius $R$ and let $f$ be a complex-valued, infinitely differentiable
function on $B_{R}.$ We show that the Laplace-Fourier series of $f$ has a
holomorphic extension which converges compactly in the Lie ball $\widehat
{B_{R}}$ in the complex space $\mathbb{C}^{n}$ when one assumes a natural
estimate for the Laplace-Fourier coefficients.


The first-named author has been supported by the Greek-Bulgarian bilateral
project B-Gr17, 2005-2008. The second-named author is partially supported by
Grant BFM2003-06335-C03-03 of the D.G.I. of Spain. Both authors acknowledge
support within the project ``Institutes Partnership'' with the Alexander von
Humboldt Foundation, Bonn.

\end{abstract}

\section{Introduction}

Let $\mathbb{S}^{n-1}=\left\{  x\in\mathbb{R}^{n}:\left|  x\right|
=1\right\}  $ be the unit sphere where $\left|  x\right|  =\sqrt{x_{1}%
^{2}+...+x_{n}^{2}}$ denotes the euclidean norm of $x=\left(  x_{1}%
,...,x_{n}\right)  \in\mathbb{R}^{n}$. Furthermore, let $Y_{k,l}\left(
x\right)  ,l=1,..,a_{k},$ be a basis for the set of all harmonic homogeneous
polynomials of degree $k\geq0$ which are orthonormal with respect to the usual
scalar product
\[
\left\langle f,g\right\rangle _{\mathbb{S}^{n-1}}=\int_{\mathbb{S}^{n-1}%
}f\left(  \theta\right)  \overline{g\left(  \theta\right)  }d\theta,
\]
see \cite{ABR92}, \cite{Koun00}. For a function $f$ given on the ball
$B_{R}:=\left\{  x\in\mathbb{R}^{n}:\left|  x\right|  <R\right\}  $ we have
the the \emph{Laplace-Fourier series} of $f$ given by
\begin{equation}
\sum_{k=0}^{\infty}\sum_{l=1}^{a_{k}}f_{k,l}\left(  r\right)  Y_{k,l}\left(
\theta\right)  \label{LF}%
\end{equation}
with the \emph{Laplace-Fourier coefficients} $f_{kl},$ defined by
\begin{equation}
f_{k,l}\left(  r\right)  =\int_{\mathbb{S}^{n-1}}f\left(  r\theta\right)
Y_{k,l}\left(  \theta\right)  d\theta; \label{LFK}%
\end{equation}
see e.g. \cite{Kalf95}, \cite{Koun00}.

This paper addresses the following question: Assume that $f$ is in $C^{\infty
}\left(  B_{R}\right)  $, the set of all infinitely many times continuously
differentiable functions $f:B_{R}\rightarrow\mathbb{C}$. Under which
conditions can we conclude that the Laplace-Fourier series of $f$ provides a
holomorphic extension to a natural domain $G$ in $\mathbb{C}^{n}$ (depending
on $B_{R}$ ), say by imposing appropriate conditions on the Laplace-Fourier
coefficients $f_{kl}$ ?

Before stating our main result, let us recall some facts already observed in
\cite{Baou}:

(i) for $f\in C^{\infty}\left(  B_{R}\right)  $ the Laplace-Fourier
coefficient $f_{k,l}\left(  r\right)  $ is infinitely many times
differentiable on the interval $\left[  0,R\right]  $ and
\begin{equation}
\frac{d^{m}}{dr^{m}}f_{k,l}\left(  0\right)  =0\text{ for }m=0,...,k-1;
\label{fnull}%
\end{equation}

(ii) the function $r\longmapsto r^{-k}f_{k,l}\left(  r\right)  $ depends only
on the variable $r^{2}.$

Clearly (i) and (ii) imply the existence of a function $p_{k,l}\in C^{\infty
}\left[  0,R^{2}\right]  $ such that
\begin{equation}
p_{k,l}\left(  r^{2}\right)  =r^{-k}f_{k,l}\left(  r\right)  . \label{defp}%
\end{equation}
It is proved in \cite{Baou} that a function $f$ which is \emph{analytic} on a
neighborhood of $0$ in $\mathbb{R}^{n},$ has a holomorphic extension to a
neighborhood of $0$ in $\mathbb{C}^{n}$ if and only if there exist $t_{0}>0$
and $M>0$ such that for all $k,m\in\mathbb{N}_{0},l=1,...,,a_{k}$
\begin{equation}
\sup_{t\in\left[  0,t_{0}\right]  }\left|  \frac{d}{dt^{m}}p_{k,l}\left(
t\right)  \right|  \leq M^{k+m+1}m!. \label{BB}%
\end{equation}
Note that the last condition implies that all functions $p_{k,l}\left(
\zeta\right)  $ are holomorphic for $\left|  \zeta\right|  <1/M.$

In this paper we want to generalize the result in \cite{Baou}, which is of
local nature, to a global one. We associate to the ball $B_{R}$ in
$\mathbb{R}^{n}$ a domain $\widehat{B_{R}}$ in $\mathbb{C}^{n}$, the so-called
\emph{Lie ball }(or the \emph{classical domain} \emph{of E. Cartan of the type
IV}, see \cite[p. 59]{ACL83}, \cite{Hua}) defined by \emph{\ }
\begin{equation}
\widehat{B_{R}}:=\{z\in\mathbb{C}^{n}:\left|  z\right|  ^{2}+\sqrt{\left|
z\right|  ^{4}-\left|  q\left(  z\right)  \right|  ^{2}}<R^{2}\},
\label{Lieball}%
\end{equation}
where $\left|  z\right|  ^{2}=\left|  z_{1}\right|  ^{2}+...+\left|
z_{n}\right|  ^{2}$ for $z=\left(  z_{1},....,z_{n}\right)  \in\mathbb{C}^{n}$
and
\[
q\left(  z\right)  =z_{1}^{2}+...+z_{n}^{2}.
\]
It is not assumed that the reader is aquainted with complex analysis for the
Lie ball and we shall only need the definition given in (\ref{Lieball}). In
particular, we shall not use the fact that the Lie ball $\widehat{B_{R}}$ can
be viewed as the \emph{harmonicity hull} of the ball $B_{R}$, see \cite[p.
42]{ACL83} or \cite{Avan}. The harmonicity hull of a domain $G$ in
$\mathbb{R}^{n}$ is the unique largest domain in $\mathbb{C}^{n}$ to which
every polyharmonic function on $G$ has a holomorphic continuation, see
\cite[p. 51]{ACL83}. According to \cite[p. 54]{ACL83}, the groundwork for the
construction of harmonicity hulls was laid by N.\ Aronszajn in 1935, and
important contributions are due to P. Lelong \cite{Lelo54}. In \cite{Kise69}
Kiselman discusses hulls for elliptic operators with constant coefficients.

We are now able to formulate our two main result:

\begin{theorem}
\label{ThmMain}Let $f\in C^{\infty}\left(  B_{R}\right)  $ and let $f_{k,l}$
and $p_{k,l}$ be defined in (\ref{LFK}) and (\ref{defp}). Then $f$ has a
holomorphic extension to the Lie ball $\widehat{B_{R}}$ if and only if the
Laplace-Fourier coefficients $f_{k,l}$ are holomorphic functions on the disc
$\mathbb{D}_{R}:=\left\{  \zeta\in\mathbb{C}:\left|  \zeta\right|  <R\right\}
$ for all $k\in\mathbb{N}_{0},l=1,...,a_{k},$ and for any $0<\tau<\rho<R$
there exists a constant $C_{\rho,\tau}>0$ such that for all $k\in
\mathbb{N}_{0},l=1,...,a_{k}$%
\begin{equation}
\left|  p_{k,l}\left(  \zeta\right)  \right|  \leq C_{\rho,\tau}\frac{1}%
{\rho^{k}}\text{ for all }\left|  \zeta\right|  \leq\tau^{2}. \label{eqneu2}%
\end{equation}

\end{theorem}

\begin{theorem}
\label{ThmMain2}Suppose that $f$ is holomorphic on $\widehat{B_{R}},$ and let
$f_{k,l}$ be the Laplace-Fourier coefficients of the restriction of $f$ to
$B_{R}$ and $p_{k,l}\left(  r^{2}\right)  :=r^{-k}f_{k,l}\left(  r\right)  .$
Then the series
\begin{equation}
\sum_{k=0}^{\infty}\sum_{l=1}^{a_{k}}p_{k,l}\left(  q\left(  z\right)
\right)  Y_{k,l}\left(  z\right)  \label{aLF}%
\end{equation}
converges compactly and absolutely in $\widehat{B_{R}}$ to $f\left(  z\right)
.$
\end{theorem}

As a byproduct of Theorem \ref{ThmMain} we can prove the following well-known
result by methods based on purely classical results for Laplace-Fourier series:

(H) \emph{Every harmonic function }$f:B_{R}\rightarrow\mathbb{C}$\emph{\ has a
holomorphic extension to the Lie ball defined in (\ref{Lieball}).}

Let us mention that the apparently most natural approach for a proof of (H),
namely via multiple Taylor series, gives only a weaker result: every harmonic
function $f:B_{R}\rightarrow\mathbb{C}$ has a holomorphic extension to the
complex ball with center $0$ and radius $R/\sqrt{2},$ so to
\[
\left\{  z\in\mathbb{C}^{n}:\left|  z\right|  <R/\sqrt{2}\right\}
\subset\widehat{B_{R}},
\]
see e.g. \cite{Haym70}, and \cite{CLP02}, \cite{Fuga82} and \cite{Khav85} for
related results. Of course, from the viewpoint of harmonicity hulls, (H) is a
trivial consequence of the fact that the Lie ball defined in (\ref{Lieball})
is indeed the harmonicity hull of $B_{R}$. However, the proof of existence and
description of a harmonicity hull in \cite{ACL83}, depending on the
serio-integral representation of a polyharmonic function, is far from being
elementary. On the other hand, J. Siciak has noted in \cite{Sici74} that (H)
is a simple consequence of the important fact (due to L.K. Hua) that for any
homogeneous polynomial $f$
\begin{equation}
\max_{x\in B_{R}}\left|  f\left(  x\right)  \right|  =\max_{z\in\widehat
{B_{R}}}\left|  f\left(  z\right)  \right|  ; \label{shilov}%
\end{equation}
for a proof of the latter result see \cite[p. 115]{Avan} or \cite{Mori98}.

\section{Proof of the results}

For $z\in\mathbb{C}^{n}$ we write $z=\xi+i\eta$ with $\xi,\eta\in
\mathbb{R}^{n}.$ Let $\left\langle \xi,\eta\right\rangle =\sum_{j=1}^{n}%
\xi_{j}\eta_{j}$ be the usual scalar product on $\mathbb{R}^{n}.$ Then
$\left|  z\right|  ^{2}=\left|  \xi\right|  ^{2}+\left|  \eta\right|  ^{2}$
and
\begin{equation}
q\left(  z\right)  :=z_{1}^{2}+...+z_{n}^{2}=\left|  \xi\right|  ^{2}-\left|
\eta\right|  ^{2}+2i\left\langle \xi,\eta\right\rangle . \label{defq}%
\end{equation}
A short computation shows that
\begin{equation}
\left|  q\left(  z\right)  \right|  ^{2}=\left(  \left|  \xi\right|
^{2}-\left|  \eta\right|  ^{2}\right)  ^{2}+4\left\langle \xi,\eta
\right\rangle ^{2}\leq\left|  z\right|  ^{4}. \label{eqqq}%
\end{equation}
Let $P_{k}^{n}\left(  t\right)  $ be the Legendre polynomial of degree $k$ for
dimension $n$ with the norming condition $P_{k}^{n}\left(  1\right)  =1.$ The
addition theorem \cite{Mull66} says that for all $x,y\in\mathbb{R}^{n}$
\begin{equation}
\sum_{l=1}^{a_{k}}Y_{k,l}\left(  x\right)  Y_{k,l}^{\ast}\left(  y\right)
=\left|  x\right|  ^{k}\left|  y\right|  ^{k}\frac{a_{k}}{\omega_{n-1}}%
P_{k}^{n}\left(  \left\langle \frac{x}{\left|  x\right|  },\frac{y}{\left|
y\right|  }\right\rangle \right)  . \label{eqADD}%
\end{equation}
Here $Y_{k,l}^{\ast}$ is the polynomial defined by conjugating the
coefficients of $Y_{k,l}$ and $a_{k}$ is the dimension of the space of all
harmonic polynomials of degree $k$, or explicitly $a_{k}:=\left(
2k+n-2\right)  \Gamma\left(  k+n-2\right)  /\Gamma\left(  k+1\right)
\Gamma\left(  n-1\right)  $ where $\Gamma$ is the Gamma function. Further
$\omega_{n-1}$ denotes the surface area of $\mathbb{S}^{n-1}.$

\begin{theorem}
Let $d_{k}$ be the leading coefficient of the Legrende polynomial $P_{k}^{n}.
$ For all $z\in\mathbb{C}^{n}$ with $\left|  q\left(  z\right)  \right|
\neq0$ the identity
\begin{equation}
\sum_{l=1}^{a_{k}}\left|  Y_{k,l}\left(  z\right)  \right|  ^{2}=\frac{a_{k}%
}{\omega_{n-1}}\left|  q\left(  z\right)  \right|  ^{k}P_{k}^{n}\left(
\frac{\left|  z\right|  ^{2}}{\left|  q\left(  z\right)  \right|  }\right)
\label{add1}%
\end{equation}
holds, and for $q\left(  z\right)  =0$
\begin{equation}
\sum_{l=1}^{a_{k}}\left|  Y_{k,l}\left(  z\right)  \right|  ^{2}=\frac{a_{k}%
}{\omega_{n-1}}d_{k}\cdot\left|  z\right|  ^{2k}. \label{add2}%
\end{equation}

\end{theorem}

\begin{proof}
Let us consider the case that $k$ is even, say $k=2k_{1}.$ Then $P_{k}%
^{n}\left(  t\right)  $ contains only even powers in $t$, say $P_{k}%
^{n}\left(  t\right)  =\sum_{s=0}^{k_{1}}$ $c_{s}t^{2s}.$ The expression on
the right hand side in (\ref{eqADD}) is equal to the polynomial
\[
S_{k}\left(  x,y\right)  :=\frac{a_{k}}{\omega_{n-1}}\sum_{s=0}^{k_{1}}%
c_{s}\left|  x\right|  ^{2\left(  k_{1}-s\right)  }\left|  y\right|
^{2\left(  k_{1}-s\right)  }\left(  \left\langle x,y\right\rangle \right)
^{2s}.
\]
Clearly $z\longmapsto q\left(  z\right)  $ is the holomorphic extension of
$x\longmapsto\left|  x\right|  ^{2}.$ Further $f\left(  z,w\right)
:=\sum_{j=1}^{n}z_{j}w_{j}$ is the holomorphic extension of $\left(
x,y\right)  \longmapsto\left\langle x,y\right\rangle .$ Hence $S_{k}$
possesses the holomorphic extension
\begin{equation}
S_{k}\left(  z,w\right)  =\frac{a_{k}}{\omega_{n-1}}\sum_{s=0}^{k_{1}}%
c_{s}\left[  q\left(  z\right)  \right]  ^{\left(  k_{1}-s\right)  }\left[
q\left(  w\right)  \right]  ^{\left(  k_{1}-s\right)  }\left[  f\left(
z,w\right)  \right]  ^{2s}. \label{eqScomplex}%
\end{equation}
Clearly $\sum_{l=1}^{a_{k}}Y_{k,l}\left(  z\right)  Y_{k,l}^{\ast}\left(
w\right)  $ is the holomorphic extension of the left hand side in
(\ref{eqADD}), so the latter expression is equal to (\ref{eqScomplex}). Now
take $w:=\overline{z}$ and note that $Y_{k,l}^{\ast}\left(  \overline
{z}\right)  =\overline{Y_{k,l}\left(  z\right)  }$. Then
\[
\sum_{l=1}^{a_{k}}\left|  Y_{k,l}\left(  z\right)  \right|  ^{2}=\frac{a_{k}%
}{\omega_{n-1}}\sum_{s=0}^{k_{1}}c_{s}\left|  q\left(  z\right)  \right|
^{2\left(  k_{1}-s\right)  }\left|  z\right|  ^{4s}.
\]
If $q\left(  z\right)  =0$ we obtain (\ref{add2}). If $\left|  q\left(
z\right)  \right|  \neq0$ we can write
\[
\sum_{l=1}^{a_{k}}\left|  Y_{k,l}\left(  z\right)  \right|  ^{2}=\frac{a_{k}%
}{\omega_{n}}\left|  q\left(  z\right)  \right|  ^{2k_{1}}\sum_{s=0}^{k_{1}%
}c_{s}\left(  \frac{\left|  z\right|  ^{2}}{\left|  q\left(  z\right)
\right|  }\right)  ^{2s}=\frac{a_{k}}{\omega_{n}}\left|  q\left(  z\right)
\right|  ^{k}P_{k}^{n}\left(  \frac{\left|  z\right|  ^{2}}{\left|  q\left(
z\right)  \right|  }\right)  .
\]
If $k$ is odd then $P_{k}^{n}\left(  t\right)  $ contains only odd powers in
$t,$ and one can employ similar techniques as in the even case. The proof is complete.
\end{proof}

\begin{lemma}
Suppose that for a $z\in\mathbb{C}^{n}$ holds$\sqrt{\left|  z\right|
^{4}-\left|  q\left(  z\right)  \right|  ^{2}}\leq\tau^{2}-\left|  z\right|
^{2}$ $.$ Then
\begin{equation}
\left|  q\left(  z\right)  \right|  ^{k}P_{k}^{n}\left(  \frac{\left|
z\right|  ^{2}}{\left|  q\left(  z\right)  \right|  }\right)  )\leq\tau^{2k}
\label{add3}%
\end{equation}

\end{lemma}

\begin{proof}
We use the Laplace representation for the Legrende polynomial in \cite[p.
21]{Mull66}, showing that for real $x\geq1$
\[
P_{k}^{n}\left(  x\right)  =\frac{1}{\omega_{n-1}}\int_{S^{n-2}}\left[
x+\sqrt{x^{2}-1}\left\langle \xi_{n-1},\eta_{n-1}\right\rangle \right]
^{k}d\eta_{n-1}.
\]
It follows that $P_{k}^{n}\left(  x\right)  \leq\left(  x+\sqrt{x^{2}%
-1}\right)  ^{k}$ for $x\geq1.$ We apply this to $x:=\left|  z\right|
^{2}/\left|  q\left(  z\right)  \right|  \geq1.$ Then
\[
\left|  q\left(  z\right)  \right|  ^{k}P_{k}^{n}(\frac{\left|  z\right|
^{2}}{\left|  q\left(  z\right)  \right|  })\leq\left(  \left|  z\right|
^{2}+\sqrt{\left|  z\right|  ^{4}-\left|  q\left(  z\right)  \right|  ^{2}%
}\right)  ^{k}.
\]
Since $\sqrt{\left|  z\right|  ^{4}-\left|  q\left(  z\right)  \right|  ^{2}%
}\leq\tau^{2}-\left|  z\right|  ^{2}$ we obtain the desired inequality.
\end{proof}

\begin{theorem}
\label{ThmLaFo}Let $p_{k,l}$ be holomorphic functions on the disc
$\mathbb{D}_{R^{2}}:=\left\{  \zeta\in\mathbb{C}:\left|  \zeta\right|
<R^{2}\right\}  $ for all $k\in\mathbb{N}_{0},l=1,...,a_{k},$ such that for
any $0<\tau<\rho<R$ there exists a constant $C_{\rho,\tau}>0$ such that for
all $k\in\mathbb{N}_{0},l=1,...,a_{k}$%
\begin{equation}
\left|  p_{k,l}\left(  \zeta\right)  \right|  \leq C_{\rho,\tau}\frac{1}%
{\rho^{k}}\text{ for all }\left|  \zeta\right|  \leq\tau^{2}. \label{eqneu3}%
\end{equation}
Then
\begin{equation}
\sum_{k=0}^{\infty}\sum_{l=1}^{a_{k}}p_{k,l}\left(  q\left(  z\right)
\right)  Y_{k,l}\left(  z\right)  \label{aLF2}%
\end{equation}
converges compactly and absolutely in $\widehat{B_{R}}$ .
\end{theorem}

\begin{proof}
We shall show that
\begin{equation}
f_{M}\left(  z\right)  :=\sum_{k=0}^{M}\sum_{l=1}^{a_{k}}p_{k,l}\left(
q\left(  z\right)  \right)  Y_{k,l}\left(  z\right)  \label{deffm}%
\end{equation}
converges compactly and absolutely on $\widehat{B_{R}}$ to a holomorphic
function. Note that $z\in\widehat{B_{R}}$ implies $\left|  z\right|  <R,$ and
(\ref{eqqq}) implies that $\left|  q\left(  z\right)  \right|  <R^{2}.$ Hence
(\ref{deffm}) is well-defined and clearly $f_{M}$ is holomorphic on
$\widehat{B_{R}}.$ Let now $K$ be a compact subset of $\widehat{B_{R}}.$
Recall that $z\in\widehat{B_{R}}$ if and olny if $\sqrt{\left|  z\right|
^{4}-\left|  q\left(  z\right)  \right|  ^{2}}<R^{2}-\left|  z\right|  ^{2}.$
Using the compactness of $K$, we see that there exists $\tau<R$ such that
$K\subset\widehat{B_{\tau}}.\ $As above it follows that $\left|  q\left(
z\right)  \right|  <\tau^{2}$ for $z\in K.$ Take $\rho$ with $\tau<\rho<R.$
Using (\ref{eqneu3}) and the fact that $q\left(  z\right)  \in\mathbb{D}%
_{\tau^{2}}$, we obtain for all $z\in K$
\[
\left|  f_{M}\left(  z\right)  \right|  \leq C_{\rho,\tau}\sum_{k=0}^{M}%
\frac{1}{\rho^{k}}\sum_{l=1}^{a_{k}}\left|  Y_{k,l}\left(  z\right)  \right|
.
\]
Note that by the Cauchy-Schwarz inequality $\sum_{l=1}^{a_{k}}\left|
Y_{k,l}\left(  z\right)  \right|  \leq\sqrt{a_{k}}\sqrt{\sum_{l=1}^{a_{k}%
}\left|  Y_{k,l}\left(  z\right)  \right|  ^{2}}$. Then (\ref{add1}) and
(\ref{add3}) show that
\[
\left|  f_{M}\left(  z\right)  \right|  \leq C_{\rho,\tau}\frac{1}%
{\sqrt{\omega_{n-1}}}\sum_{k=0}^{M}a_{k}\frac{\tau^{k}}{\rho^{k}}.
\]
Since $a_{k}/a_{k+1}$ converge to $1$ we see that (\ref{deffm}) converges. It
follows that $f_{M}$ converges uniformly to a holomorphic function.
\end{proof}

The following result was proved in \cite{Baou}.

\begin{lemma}
\label{Lem4} Let $f\in C^{\infty}\left(  B_{R}\right)  .$ Then the
Laplace-Fourier coefficients $f_{k,l}$ are infinitely times differentiable at
$0$ and $\frac{d^{m}}{dr^{m}}f_{k,l}\left(  0\right)  =0$ for $m=0,...,k-1.$
\end{lemma}

\begin{theorem}
\label{ThmMain3}Let $f\in C^{\infty}\left(  B_{R}\right)  $ and let
$p_{k,l}\left(  r\right)  :=r^{-k}f_{k,l}\left(  r\right)  .$ Suppose that
there exists a continuous function $\widetilde{f}:\Delta_{R}\times
\mathbb{S}^{n-1}\rightarrow\mathbb{C}$ such that

1) $\widetilde{f}\left(  r,\theta\right)  =f\left(  r\theta\right)  $ for all
$0\leq r<R$ and $\theta\in\mathbb{S}^{n-1}$

2) for each $\theta\in\mathbb{S}^{n-1}$ the function $\zeta\longmapsto
\widetilde{f}\left(  \zeta,\theta\right)  $ is holomorphic for $\left|
\zeta\right|  <R.$

Then $f$ has a holomorphic extension to the Lie ball, the Laplace-Fourier
series in (\ref{aLF2}) converges compactly to $f$ and the coefficients
$p_{k,l}\left(  \zeta\right)  $ satisfy (\ref{eqneu3}).
\end{theorem}

\begin{proof}
We want to apply Theorem \ref{ThmLaFo}, and we consider at first the
Laplace-Fourier coefficients $f_{k,l}$ of $f.$ We define the holomorphic
extension of $f_{k,l}\left(  r\right)  $ by
\begin{equation}
f_{k,l}\left(  \zeta\right)  =\int_{\mathbb{S}^{n-1}}\widetilde{f}\left(
\zeta,\theta\right)  Y_{k,l}\left(  \theta\right)  d\theta. \label{deffkl}%
\end{equation}
Let us show that $f_{k,l}$ is indeed holomorphic for $\left|  \zeta\right|
<R:$ let $\left|  \zeta_{0}\right|  <R$ and $\zeta_{m}\rightarrow\zeta_{0}$
with $\left|  \zeta_{m}\right|  <R.$ Let $\rho>0$ such that $\left|  \zeta
_{m}\right|  <\rho$ for all $m.$ Since $\zeta\longmapsto f\left(  \zeta
,\theta\right)  $ is holomorphic for $\left|  \zeta\right|  <R$ we can use
Cauchy's integral formula on the path $\gamma_{\rho}\left(  t\right)  =$ $\rho
e^{it},$ i.e.
\[
\widetilde{f}\left(  \zeta,\theta\right)  =\frac{1}{2\pi i}\int_{\gamma_{\rho
}}\frac{\widetilde{f}\left(  \xi,\theta\right)  }{\xi-\zeta}d\xi.
\]
We have to show that the differential quotient $D_{m}:=\frac{f_{k,l}\left(
\zeta_{m}\right)  -f_{k,l}\left(  \zeta_{0}\right)  }{\zeta_{m}-\zeta_{0}}$
converges. Clearly
\[
D_{m}=\frac{1}{2\pi i}\frac{1}{\zeta_{m}-\zeta_{0}}\int_{\mathbb{S}^{n-1}}%
\int_{B_{\rho}}\left[  \frac{\widetilde{f}\left(  \xi,\theta\right)  }%
{\xi-\zeta_{m}}-\frac{\widetilde{f}\left(  \xi,\theta\right)  }{\xi-\zeta_{0}%
}\right]  Y_{k,l}\left(  \theta\right)  d\theta
\]
and this is equal to
\[
\frac{1}{2\pi i}\int_{\mathbb{S}^{n-1}}\int_{B_{\rho}}\frac{\widetilde
{f}\left(  \xi,\theta\right)  }{\left(  \xi-\zeta_{m}\right)  \left(
\xi-\zeta_{0}\right)  }Y_{k,l}\left(  \theta\right)  d\theta.
\]
Since this expression has clearly a limit, $f_{k,l}$ is holomorphic. Further
$p_{k,l}\left(  r^{2}\right)  :=r^{-k}f_{k,l}\left(  r\right)  $ has a
holomorphic extension to $\mathbb{D}_{R^{2}}$ by Lemma \ref{Lem4}. Let now
$0<\tau<\rho<R$. By the Cauchy-Schwarz inequality and the orthonormality of
$Y_{k,l}\left(  \theta\right)  $ one obtains from (\ref{deffkl}) for $\left|
\zeta\right|  =\rho$
\[
\left|  f_{k,l}\left(  \zeta\right)  \right|  ^{2}\leq\int_{\mathbb{S}^{n-1}%
}\left|  f\left(  \zeta\theta\right)  \right|  ^{2}d\theta\cdot\int
_{\mathbb{S}^{n-1}}\left|  Y_{k,l}\left(  \theta\right)  \right|  ^{2}%
d\theta\leq\omega_{n-1}\max_{t\in\left[  0,2\pi\right]  ,\theta\in
\mathbb{S}^{n-1}}\left|  f\left(  e^{it}\rho\theta\right)  \right|  ^{2}.
\]
The Cauchy estimate $\left|  g^{\left(  s\right)  }\left(  0\right)  \right|
\leq\frac{s!}{\rho^{s}}\max_{\left|  \zeta\right|  =\rho}\left|  g\left(
\zeta\right)  \right|  $ applied to the function $g=f_{k,l},$ and $s=k+m,$ and
the last estimate imply
\begin{equation}
\left|  \frac{d^{m+k}}{dz^{m+k}}f_{k,l}\left(  0\right)  \right|  \leq
\sqrt{\omega_{n-1}}\max_{t\in\left[  0,2\pi\right]  ,\theta\in\mathbb{S}%
^{n-1}}\left|  f\left(  e^{it}\rho\theta\right)  \right|  \cdot\frac{\left(
k+m\right)  !}{\rho^{m+k}}. \label{eqlast1}%
\end{equation}
Let us write $f_{k,l}\left(  r\right)  =\sum_{m=k}^{\infty}\frac{1}{m!}%
\frac{d^{m}}{dr^{m}}f_{k,l}\left(  0\right)  \cdot r^{m}$ for $0\leq r<R,$ cf.
(\ref{fnull}). Since $r^{-k}f_{k,l}\left(  r\right)  $ is an even function we
can write
\[
r^{-k}f_{k,l}\left(  r\right)  =\sum_{m=0}^{\infty}\frac{1}{\left(
k+2m\right)  !}\frac{d^{2m+k}}{dr^{2m+k}}f_{k,l}\left(  0\right)  \cdot
r^{2m}.
\]
Hence
\[
p_{k,l}\left(  t\right)  =\sum_{m=0}^{\infty}\frac{1}{\left(  k+2m\right)
!}\frac{d^{2m+k}}{dr^{2m+k}}f_{k,l}\left(  0\right)  \cdot t^{m}%
\]
and using (\ref{eqlast1}) we obtain the estimate for $\left|  \zeta\right|
\leq\tau^{2},$
\[
\left|  p_{k,l}\left(  \zeta\right)  \right|  \leq C\sum_{m=0}^{\infty}%
\frac{\tau^{2m}}{\rho^{2m+k}}=\frac{C}{\rho^{k}}\frac{1}{1-\frac{\tau^{2}%
}{\rho^{2}}}%
\]
where $C:=\sqrt{\omega_{n-1}}\max_{t\in\left[  0,2\pi\right]  ,\theta
\in\mathbb{S}^{n-1}}\left|  f\left(  e^{it}\rho\theta\right)  \right|  $ does
not depend on $k\in\mathbb{N}_{0}$ and $l=1,...,a_{k}.$ Now we apply Theorem
\ref{ThmLaFo}, which gives that the Laplace-Fourier series in (\ref{aLF2})
converges compactly to a holomorphic function $g.$ From the uniform
convergence of the Laplace-Fourier seeries of $g$ it is easy to see that $g$
has the same Laplace-Fourier coefficients as $f,$ so we conclude that $f=g.$
\end{proof}

\begin{proof}
of Theorem \ref{ThmMain2}. Let $f$ be holomorphic on $\widehat{B_{R}}.$
Clearly $g:=f\mid B_{R}$ is a $C^{\infty}-$function. Let us define a function
$\widetilde{g}:\Delta_{R}\times\mathbb{S}^{n-1}\rightarrow\mathbb{C}$ by
$\widetilde{g}\left(  \zeta,\theta\right)  :=f\left(  \zeta\theta\right)  $
which is well-defined since $\zeta\theta\in\widehat{B_{R}}$ for $\left|
\zeta\right|  <R$ and $\theta\in\mathbb{S}^{n-1}.$ Then $\widetilde{g}$ is
continuous, and for each fixed $\theta\in\mathbb{S}^{n-1},$ the function
$\zeta\longmapsto\widetilde{g}\left(  \zeta,\theta\right)  $ is holomorphic
and $\widetilde{g}\left(  r\theta\right)  =f\left(  r\theta\right)  =g\left(
r\theta\right)  $ for $0\leq r<R$ and for all $\theta\in\mathbb{S}^{n-1}.$ The
result follows from Theorem \ref{ThmMain3}.
\end{proof}

\begin{proof}
{of Theorem \ref{ThmMain}.} Let $f$ be holomorphic on $\widehat{B_{R}}.$ By
the previous proof we can apply Theorem \ref{ThmMain3}, so (\ref{eqneu2}) is
satisfied. For the converse, let $f\in C^{\infty}\left(  B_{R}\right)  $ and
assume that (\ref{eqneu2}) holds. By Theorem \ref{ThmLaFo} the Laplace-Fourier
series in (\ref{aLF}) converges compactly to a holomorphic function $g.$
Clearly $g$ has the same Laplace-Fourier coefficients as $f,$ so we conclude
that $f=g.$
\end{proof}

For $f\in C^{\infty}\left(  B_{R}\right)  $ let $f_{m}$ be the $m$-th
\emph{homogeneous} Taylor polynomial, so
\[
f_{m}\left(  x\right)  :=\sum_{\left|  \alpha\right|  =m}\frac{D^{\alpha
}f\left(  0\right)  }{\alpha!}x^{\alpha}%
\]
where we use the standard multi-index notation $x^{\alpha}=x_{1}^{\alpha_{1}%
}...x_{n}^{\alpha_{n}}$ for $x=\left(  x_{1},...,x_{n}\right)  \in
\mathbb{R}^{n}$ and $\alpha=\left(  \alpha_{1},...,\alpha_{n}\right)
\in\mathbb{N}_{0}^{n}$. We consider the space $A\left(  B_{R}\right)  $ of all
$f\in C^{\infty}\left(  B_{R}\right)  $ such that
\begin{equation}
f\left(  x\right)  =\sum_{m=0}^{\infty}f_{m}\left(  x\right)  \label{eqf}%
\end{equation}
converges absolutely and uniformly on compact subsets of $B_{R}$. The space
$A\left(  B_{R}\right)  $ evolved naturally in the investigations in
\cite{rend} for solving a conjecture of W. Hayman concerning sets of
uniqueness for polyharmonic functions, see \cite{Haym94}. It is well known
that every harmonic function $f:B_{R}\rightarrow\mathbb{C}$ is in $A\left(
B_{R}\right)  .$ We shall now prove that each $f\in A\left(  B_{R}\right)
$\ possesses a holomorphic extension to $\widehat{B_{R}}$ (as we said in the
introduction, this fact was already noticed in \cite{Sici74} with a different proof).

\begin{theorem}
Let $f\in C^{\infty}\left(  B_{R}\right)  ,$ and suppose that the homogeneous
Taylor series $\sum_{m=0}^{\infty}f_{m}$ converges compactly in $B_{R}$ to
$f.$ Then $f$ has a holomorphic extension to the Lie ball.
\end{theorem}

\begin{proof}
Let $f$ as described in the theorem. Write $x\in B_{R}$ in polar coordinates
$x=r\theta$ with $\theta\in\mathbb{S}^{n-1}$ and $r\geq0.$ Then $f_{m}\left(
r\theta\right)  =r^{m}f_{m}\left(  \theta\right)  $ and $\sum_{m=0}^{\infty
}r^{m}\left|  f_{m}\left(  \theta\right)  \right|  $ converges uniformly for
all $\theta\in\mathbb{S}^{n-1}$ and $0\leq r<R.$ Let us define $\widetilde
{f}:\Delta_{R}\times\mathbb{S}^{n-1}\rightarrow\mathbb{C}$ by
\[
\widetilde{f}\left(  \zeta,\theta\right)  :=\sum_{m=0}^{\infty}\zeta^{m}%
f_{m}\left(  \theta\right)  .
\]
Clearly $\widetilde{f}$ is continuous, and the map $\zeta\longmapsto
\widetilde{f}\left(  \zeta\theta\right)  $ for $\left|  \zeta\right|  <R$ is
clearly holomorphic for each fixed $\theta\in\mathbb{S} ^{n-1}.$ The result
follows from Theorem \ref{ThmMain3}.
\end{proof}

\begin{acknowledgement}
Both authors acknowledge support within the project 
Institutes Partnership with the Alexander von Humboldt
Foundation, Bonn. The second author was partially supported by Grant
BFM2003-06335-C03-03 of the D.G.I. of Spain. 

\end{acknowledgement}

Addresses: 

 O. Kounchev\\Institute of Mathematics and Informatics, Bulgarian Academy of Sciences,  8
Acad. G. Bonchev Str., 1113 Sofia, Bulgaria; kounchev@math.bas.bg

 H. Render\\Departamento de Matem\'{a}ticas y Computaci\'{o}n, Universidad de La Rioja, 
Edificio Vives, Luis de Ulloa s/n., 26004 Logro\~{n}o, Spain; render@gmx.de

\end{document}